\documentclass[12pt,a4paper]{amsart}

\usepackage[latin1]{inputenc} 

\usepackage{enumerate}

\usepackage{rotating}
\usepackage{fancybox}
\usepackage{float}
\restylefloat{table}

\usepackage{graphicx}

\usepackage{amsmath}

\usepackage{amssymb}

\usepackage{amsthm}

\usepackage[arrow, matrix, curve]{xy}

\usepackage{a4wide}

\usepackage{txfonts}

\usepackage{hyperref}

\theoremstyle{plain}
\newtheorem{thm}{Theorem}[section]

\newtheorem{cor}[thm]{Corollary}

\theoremstyle{definition}
\newtheorem{defn}[thm]{Definition}
\newtheorem{lem}[thm]{Lemma}

\newtheorem{rmk}[thm]{Remark}

\numberwithin{equation}{thm}

\newcommand{\emphbf}[1]{\emph{\textbf{#1}}}

\DeclareMathAlphabet{\mathpzc}{OT1}{pzc}{m}{it}

\newcommand{\rad}[1]{\radoperator(#1)}
\newcommand{\radsquare}[1]{\radoperator^2(#1)}

\DeclareMathOperator{\radoperator}{rad}
\DeclareMathOperator{\Kopf}{top}
\DeclareMathOperator{\soc}{soc}

\DeclareMathOperator{\Hom}{Hom}

\DeclareMathOperator{\modu}{mod}

\begin{document}

\title{Biserial algebras via subalgebras and the path algebra of $\mathbb{D}_4$}
\author{Julian K\"ulshammer}
\address{Christian-Albrechts-Universit\"at zu Kiel, Ludewig-Meyn-Str. 4, 24098 Kiel, Germany}
\email{kuelshammer@math.uni-kiel.de}

\begin{abstract}
We give two new criteria for a basic algebra to be biserial. The first one states that an algebra is biserial iff all subalgebras of the form $eAe$ where $e$ is supported by at most $4$ vertices are biserial. The second one gives some condition on modules that must not exist for a biserial algebra. These modules have properties similar to the module with dimension vector $(1,1,1,1)$ for the path algebra of the quiver $\mathbb{D}_4$.\\
Both criteria generalize criteria for an algebra to be Nakayama. They rely on the description of a basic biserial algebra in terms of quiver and relations given by R. Vila-Freyer and W. Crawley-Boevey \cite{CBVF98}.
\end{abstract}

\maketitle

\section{Introduction}
Throughout this paper let $k$ be an algebraically closed field, denote by $A$ a fi\-nite di\-men\-sio\-nal $k$-algebra, its (Jacobson) radical by $\rad{A}$ and by $\modu A$ the category of all finitely generated left modules. For $M\in \modu A$ we denote by $\radoperator^i M$ the $i$-th radical of $M$, by $\soc^i M$ the $i$-th socle of $M$ (cf. \cite{B95} definition 1.2.1) and by $Q=(Q_0,Q_1,s,t)$ a quiver with set of vertices $Q_0$, set of arrows $Q_1$ and starting (resp. terminal) point functions $s$ (resp. $t$). For every point $i\in Q_0$ of the quiver there exists a zero path, denoted by $e_i$, the ideal of the path algebra $kQ$ generated by the arrows will be denoted by $kQ^+$. For basic facts on radical, socle and quivers, that we use without further reference, we refer to \cite{ASS06}.\\
In 1979 K. Fuller (\cite{F79}) defined biserial algebras as algebras whose indecomposable projective left and right modules have uniserial submodules which intersect zero or simple and which sum to the unique maximal submodule (Tachikawa mentioned this condition before, but didn't give these algebras a name \cite{T61} proposition 2.7). These natural generalizations of Nakayama algebras are a class of tame algebras as W. Crawley-Boevey showed in \cite{CB95}. Examples of these algebras are blocks of group algebras with cyclic or dihedral defect group (see e.g. \cite{Ri75}, \cite{E87}), the algebras appearing in the Gel'fand-Ponomarev classification of the singular Harish-Chandra modules over the Lorentz group (\cite{GP68}) as well as special biserial algebras, which were recently used to test certain conjectures (\cite{EHIS04}, \cite{LM04}, \cite{S10}).\\
As one looks at Nakayama algebras (cf. \cite{ASS06} Section V.3) there are at least three ways to describe them: First via the projective left and right modules, i.e. they are uniserial, second via the (ordinary) quiver (and its relations), i.e. the quiver of $A$ is a linearly oriented (extended) Dynkin diagram of type $\mathbb{A}_n$ or $\tilde{\mathbb{A}}_n$ for some $n\geq 1$, and third via certain ``small'' modules in the module category (cf. Lemma \ref{lem2}), i.e. there exists no local module $M$ of Loewy length two, such that $l(\rad M)=2$ and no colocal module $M$ of Loewy length two, such that $l(M/\soc M)=2$ (we could call this property non-linearly oriented $\mathbb{A}_3$-freeness).\\
For biserial algebras aside from the original definition a description of basic biserial algebras in terms of quivers and relations is due to R. Vila-Freyer and W. Crawley-Boevey (\cite{CBVF98}). We use this description to obtain one in terms of certain ``small'' modules analogous to the description for Nakayama algebras given above.\\
A basic algebra $A$ will be called $\mathbb{D}_4$-free iff there is no $A$-module with similar properties to the one with dimension vector $(1,1,1,1)$ for the path algebra of the quiver $\mathbb{D}_4$. Our result will then be the following:

\begin{thm}\label{thm1}
A basic algebra $A$ is biserial iff it is $\mathbb{D}_4$-free.
\end{thm}

Furthermore, from the description of the quiver of $A$ we can see that it is necessary and sufficient that all subalgebras of the form $eAe$ with support of one vertex and its neighbouring vertices are Nakayama. We could call these subalgebras of type $\mathbb{A}_3$. Our second main result generalizes this for biserial algebras and states

\begin{thm}\label{thm2}
An algebra $A$ is biserial iff all subalgebras $eAe$ of type $\mathbb{D}_4$, that is, with support of a vertex and at most three of its neighbouring vertices, are biserial.
\end{thm}

Our paper is organized as follows. In Section 2, we recall the results of \cite{VF94} and \cite{CBVF98} giving a description of a basic biserial algebra in terms of its quiver and relations. Section 3 then gives the precise statement of Theorem \ref{thm2} and its proof. The precise definition of $\mathbb{D}_4$-free and the proof of Theorem \ref{thm1} is then presented in Section 4.

\section{Biserial algebras}

\begin{defn}[\cite{F79}]
An algebra $A$ is called \emphbf{biserial} if for every projective left or right module $P$ there exist uniserial submodules $U$ and $V$ of $P$ satisfying $\rad{P}=U+V$ (not necessarily a direct sum), such that $U\cap V$ is zero or simple.
\end{defn}

In the remainder of this section we present the results of R. Vila-Freyer and W. Crawley-Boevey who describe biserial algebras in terms of quivers and relations. For the proofs we refer to \cite{CBVF98}. The notation has been adjusted to ours.

\begin{defn}[\cite{CBVF98} Definitions 1-3]
\begin{enumerate}[(i)]
\item A \emphbf{bisection} of a quiver $Q$ is a pair $(\sigma, \tau)$ of functions $Q_1\to \{\pm 1\}$, such that if $a$ and $b$ are distinct arrows with $s(a)=s(b)$ (respectively $t(a)=t(b)$), then $\sigma(a)\neq \sigma(b)$ (respectively $\tau(a)\neq \tau(b)$). A quiver, which admits a bisection, i.e. in each vertex there start and end at most two arrows, is called \emphbf{biserial}.
\item Let $Q$ be a quiver and $(\sigma,\tau)$ a bisection. We say that a path $a_r\cdots a_1$ in $Q$ is a \emphbf{good path}, or more precisely is a $(\sigma,\tau)$\emphbf{-good path},  if $\sigma(a_{i})=\tau(a_{i-1})$ for all $1< i\leq r$. Otherwise we say that it is a \emphbf{bad path}, or is a $(\sigma,\tau)$\emphbf{-bad path}. The paths $e_i$ ($i\in Q_0$) are good.
\item By a \emphbf{bisected presentation} $(Q,\sigma,\tau, \mathpzc{p},\mathpzc{q})$ of an algebra $A$ we mean that $Q$ is a quiver with a bisection $(\sigma,\tau)$ and that $\mathpzc{p}, \mathpzc{q}: kQ\to A$ are surjective algebra homomorphisms with $\mathpzc{p}(e_i)=\mathpzc{q}(e_i)$ for all $i\in Q_0$, $\mathpzc{p}(a), \mathpzc{q}(a)\in \rad{A}$ for all arrows $a\in Q_1$ and $\mathpzc{q}(a)\mathpzc{p}(x)=0$ whenever $a,x\in Q_1$ with $ax$ a bad path.
\end{enumerate}
\end{defn}

\begin{thm}[\cite{CBVF98} Theorem]
Any basic biserial algebra $A$ has a bisected presentation $(Q,\sigma,\tau,\mathpzc{p},\mathpzc{q})$ in which $Q$ is the quiver of $A$. Conversely any algebra with a bisected presentation is basic and biserial.
\end{thm}

\begin{cor}[\cite{CBVF98} Corollary 3]\label{corVF}
Suppose that $Q$ is a quiver, $(\sigma, \tau)$ is a bisection, elements $d_{ax}\in kQ$ are defined for each bad path $ax$, $a,x\in Q_1$, and they satisfy
\begin{enumerate}[({C}1)]
\item Either $d_{ax}=0$ or $d_{ax}=\omega b_t\cdots b_1$ with $\omega\in k^\times, t\geq 1$ and $b_t\cdots b_1x$ a good path with $t(b_r)=t(a)$ and $b_t\neq a$,
\item if $d_{ax}=\phi b$ and $d_{by}=\psi a$ with $\phi,\psi\in k^\times$ and $a,b,x,y\in Q_1$, then $\phi\psi\neq 1$.
\end{enumerate}
If $I$ is an admissible ideal in $kQ$ which contains all the elements $(a-d_{ax})x$, then $kQ/I$ is a basic biserial algebra. Conversely for every basic biserial algebra $A$ there exist a quiver $Q$, a bisection $(\sigma,\tau)$ and for every bad path $ax$, $a,x\in Q_1$, elements $d_{ax}$, which satisfy the above conditions, and an admissible ideal $I$, such that $A\cong kQ/I$.
\end{cor}

Observe that the algebras where $d_{ax}=0$ for all bad paths $ax$ are precisely the special biserial algebras which are a lot better understood.\par
The following technical lemma will be used in the next theorem. Its proof relies on Lemma 1.2 in \cite{CBVF98}. The remaining parts are proved by similar methods, so we omit it here although it is nowhere published.

\begin{lem}[\cite{VF94}  Lemma 2.1.3.1]\label{tl}
Let $A=kQ/I$ as in Corollary \ref{corVF} and let $a,x\in Q_1$ be arrows, such that $ax$ is a bad path and $d_{ax}=\omega b_t\cdots b_1$ with $\omega\in k^\times, b_t,\dots,b_1\in Q_1$. Then for any arrow $d$ with $s(d)=t(a)$ we have $dax$ and $db_t\cdots b_1x$ are both elements of $I$.
\end{lem}

\section{Subalgebras of type $\mathbb{D}_4$}

As a first application of the description due to R. Vila-Freyer and W. Crawley-Boevey, the next theorem tells us that we can restrict ourselves to algebras whose quiver has at most 4 vertices and one vertex is connected to all the others by at least one arrow.\\
For an easier statement of our first main result, we introduce here two sets of neighbours of some given vertex. These sets will correspond via idempotents $e$ to subalgebras $eAe$ of $A$ that one can use to test the biseriality of $A$.

\begin{defn}
Let $A=kQ/I$ and let $l\in Q_0$.
\begin{enumerate}[(i)]
\item Then $N(l):=\{j\neq l|j\thickspace \text{is connected to $l$ by at least one arrow in the quiver of}\thickspace A\}$ is called the \emphbf{set of neighbouring vertices of $l$}.
\item If $|N(l)|<4$, then define $J(l):=N(l)$ and if $N(l)=4$, then call any subset $J(l)\subset N(l)$ with $|J(l)|=3$ a \emphbf{set of neighbours of $l$ of type $\mathbb{D}_4$}.
\end{enumerate}
\end{defn}

\begin{thm}\label{5pt}
\begin{enumerate}[(i)]
\item Let $A$ be a biserial algebra. Then the algebra $eAe$ is biserial for every idempotent $e\in A$.
\item Let $A=kQ/I$ be a basic algebra with zero paths $e_1,\dots,e_n$. Then $A$ is biserial iff for all idempotents $e\in A$ of the form $e=e_l+\sum_{j\in N(l)}e_j$ the algebra $eAe$ is biserial.
\item Let $A=kQ/I$ be a basic algebra with zero paths $e_1,\dots,e_n$. Then $A$ is biserial iff for all idempotents $e\in A$ of the form $e=e_l+\sum_{j\in J(l)}e_j$ for some set of neighbours of $l$ of type $\mathbb{D}_4$ the algebra $eAe$ is biserial.
\end{enumerate}
\end{thm}

\begin{proof}
We can assume without loss of generality that $A$ is a basic algebra. First assume that $A$ is biserial. We want to show that for all idempotents $e\in A$ the algebra $eAe$ is biserial. Therefore let $e=e_1+\dots+e_k$ be a decomposition of $e$ into primitive orthogonal idempotents and analogously $1-e=e_{k+1}+\dots+e_n$. Then $A\cong kQ/I$, where the idempotents $e_1,\dots,e_n$ correspond to the zero paths and $I$ satisfies the conditions of Corollary \ref{corVF} and $eAe\cong ekQe/eIe$.\\
It is a standard result that one can check quite easily, that $\{e_1,\dots,e_k\}$ is a complete set of primitive orthogonal idempotents for $eAe$. The radical of $eAe$ is $e\rad{A}e$ since this is a nilpotent ideal and one can use $\Hom$-functors of projective modules to get from the sequence $0\to \rad{A}\to A\to A/\rad{A}\to 0$ to the sequence $0\to e\rad{A}e\to eAe\to eA/\rad{A}e\to 0$, which is therefore short exact and the factor is semisimple. An arrow in the quiver of $eAe$ does therefore correspond to an element in $\rad{eAe}/\radsquare{eAe}=e\rad{A}e/e\rad{A}e\rad{A}e$. Note that $e\rad{A}e\rad{A}e\subseteq e\radsquare{A}e$ but in general there is no equality. Therefore there can be arrows in the quiver of $eAe$ that do not come from arrows in the quiver of $A$, but instead from longer paths that do not pass through one of the vertices $1,\dots,k$. 
Let us fix some notation: Denote by $\widetilde{a_1\dots a_s}$ the path $a_1\dots a_s$ as an element of $ekQe$ in case $1\leq s(a_s), t(a_1)\leq k$ and $k+1\leq s(a_i)\leq n$ for $1\leq i\leq s-1$. Such a path $a_1\dots a_s$ will be called irreducible in case $\widetilde{a_1\cdots a_s}\nequiv 0\mod eIe$. We now have a presentation $eAe\cong k\tilde{Q}/\tilde{I}$ where $\tilde{Q}_0=\{1,\dots,k\}$, $\tilde{Q}_1$ is the set of irreducible paths and $\tilde{I}:=eIe\cap k\tilde{Q}$ will be the induced ideal (not necessarily admissible, but $(k\tilde{Q}^+)^m\subseteq \tilde{I}\subseteq k\tilde{Q}^+$).\\
The same proof as for admissible ideals (cf. \cite{ASS06} Lemma II.2.10) shows that $\rad{k\tilde{Q}/\tilde{I}}=k\tilde{Q}^+/\tilde{I}$. So an arrow in the quiver of $eAe$ corresponds to a basis element of $(k\tilde{Q}^+/\tilde{I})/(k\tilde{Q}^+/\tilde{I})^2\cong k\tilde{Q}^+/((k\tilde{Q}^+)^2+\tilde{I})$. So $\tilde{Q}$ is in general not the quiver of $eAe$. We now want to show, that the quiver of $eAe$ is biserial and that we can choose $Q'_1\subseteq \tilde{Q}_1$ a base of $k\tilde{Q}^+/(\tilde{I}+(k\tilde{Q}^+)^2)$ in such a way, that $Q'$ inherits a bisection from $Q$ (Taking a base guarantees, that $Q'$ will be the quiver of $eAe$): In any point of $Q$ there start at most two arrows. The presence of more than two irreducible paths from a vertex $i$ to a vertex $j$, both in $\{1,\dots,k\}$ leads to two irreducible paths from $i$ to $j$ of the form $qa_sx_1p$ and $q'b_1x_1p$ for some paths $p,q,q'$ and arrows $a_s,x_1,b_1\in Q_1$, $a_s\neq b_1$.
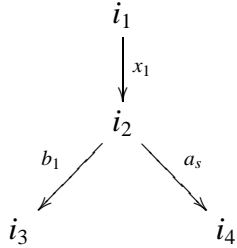
\[\begin{xy}\xymatrix{&i_1\ar[d]^{x_1}\\&i_2\ar[rd]^{a_s}\ar[ld]_{b_1}\\i_3&&i_4}\end{xy}\]
According to Corollary \ref{corVF} at any such crossing there has to be a relation, either of the form $a_sx_1$ or of the form $(a_s-\omega b_t\cdots b_1)x_1$ for some $\omega\in k^\times$, $t\geq 1$ and $b_t,\dots,b_2\in Q_1$. In the former case the path $qa_sx_1p$ belongs to $eIe$, a contradiction. In the latter case, either $j$ lies on the longer path $b_t\cdots b_1$, then $qa_sx_1p\in k\tilde{Q}^2+\tilde{I}$, a contradiction, otherwise at most one of the paths would lead to an arrow in $Q'$ as $a_sx_1\equiv \omega b_t\cdots b_1x_1\mod I$. This shows that the quiver of $eAe$ is biserial.\\
We now want to choose $Q_1'\subseteq \tilde{Q}_1$ as described above, such that $Q'$ inherits a bisection from $Q$. Assume there are more than two arrows from $i$ to $j$ in $\tilde{Q}$. Suppose two of them start with the same arrow. Then the above arguments show that they have to be linearly dependent modulo $\tilde{I}$. So for every choice $Q_1'\subseteq \tilde{Q}_1$ of a base of $k\tilde{Q}^+/((k\tilde{Q}^+)^2+\tilde{I})$ only one of them will appear, so if we define $\sigma'(\widetilde{a_1\cdots a_s}):=\sigma(a_s)$ that will consistently define one part of a bisection. If on the other hand we have two paths starting with different arrows but ending with the same path $a'$, i.e. we have the following picture
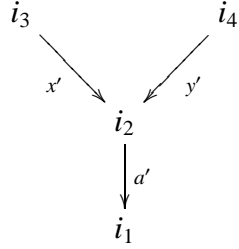
\[\begin{xy}\xymatrix{i_3\ar[rd]_{x'}&&i_4\ar[ld]^{y'}\\&i_2\ar[d]^{a'}\\&i_1}\end{xy}\]
with $x',y'\in Q_1$, $\sigma(a')\neq \tau(x')$, $x'\neq y'$ and the two different paths are $a'x'p$ and $a'y'p'$ with $p,p'\in kQ$. If the length of the path $a'$, regarded as an element of $kQ$ is greater than one, then Lemma \ref{tl} leads to the contradiction that $a'x'\in I$. If $a'\in Q_1$ then there exist arrows $b_{t'}',\dots, b_1'\in Q_1$, such that $(a'-\omega'b_{t'}'\cdots b_1')x'\in I$, so we can replace $a'x'p$ in a choice of a base by $b_{t'}'\cdots b_1'x'p$ and will also get a base of $k\tilde{Q}^+/((k\tilde{Q}^+)^2+\tilde{I})$. Then we can also define $\tau'$ consistently by $\tau'(\widetilde{a_1\cdots a_s}):=\tau(a_1)$ yielding a bisection of $Q'$ inherited from $Q$. If we take $I':=\tilde{I}\cap kQ'$ then this is an admissible ideal with $kQ'/I'\cong k\tilde{Q}/\tilde{I}\cong eAe$.\\
Now we want to show, that the necessary relations of Corollary \ref{corVF} exist. Therefore let $\tilde{a}\tilde{x}:=\widetilde{a_1\cdots a_s}\widetilde{x_1\cdots x_r}$ be a bad path of length two in $Q'$ ($a_1,\dots,a_s, x_1,\dots,x_r\in Q_1$). If $s=1$, then either $a_sx_1$ is in $I$ and therefore $\tilde{a}\tilde{x}\in I'$ or there exists $\omega\in k^\times, b_1,\dots,b_t\in Q_1$, such that $(a_1-\omega b_t\cdots b_1)x_1\in I$ and therefore $(\tilde{a}_1-\omega \widetilde{b_t\cdots b_1})\widetilde{x_1\cdots x_r}\in I'$, where $\widetilde{b_t\cdots b_1}$ is the path corresponding to $b_1\cdots b_1$ in $ekQ'e$. If otherwise $s>1$, then by Lemma \ref{tl}, $a_1\cdots a_sx_1\in I$, therefore $\tilde{a}\tilde{x}\in I'$. This shows (i).\\
For the other direction of (ii) let $A$ be an algebra, such that $eAe$ is biserial for all idempotents of the required form. For any idempotent $e_l$, there exists a bisected presentation $(Q_l,\sigma_l,\tau_l,\mathpzc{p}_l,\mathpzc{q}_l)$. Set $\mathpzc{p}(e_l):=\mathpzc{q}(e_l):=e_l$ and for arrows $a$ starting (resp. ending) at $l$ in $eAe$, that come from arrows (and not from longer paths) in $A$, set $\sigma(a):=\sigma_l(a)$ and $\mathpzc{q}(a):=\mathpzc{q}_l(a)$ (resp. $\tau(a):=\tau_l(a)$ and $\mathpzc{p}(a):=\mathpzc{p}_l(a)$). Taking idempotents of the form $e_l+\sum_{j\in N(l)}e_j$ assures that we define values of $\sigma, \tau, \mathpzc{p}, \mathpzc{q}$ for any arrow $a\in Q_1$ in a compatible way. To show that this defines a bisected presentation for $A$ it only remains to prove that $\mathpzc{p}$ and $\mathpzc{q}$ are surjective. This follows as in the construction of the quiver of $A$ (cf. \cite{ASS06} Theorem 3.7) since the elements $\overline{\mathpzc{p}(a)}$ (resp. $\overline{\mathpzc{q}(a)}$) span $A/\radsquare{A}$.\\
For the other direction of (iii) let $A$ be an algebra, such that $eAe$ is biserial for all idempotents of the required form. For vertices where there are at most three neighbouring vertices proceed as in (ii). If there are four neighbouring vertices for $l$, then there are two arrows $x,y$ ending in $l$ and two arrows $a,b$ starting at $l$. Assume without loss of generality $s(x)=j_1$, $s(y)=j_2$, $t(a)=j_3$, $t(b)=j_4$. Denote the four bisected presentations that we get for this vertex by $(Q_{l}^i, \sigma_l^i, \tau_l^i, \mathpzc{p}_l^i, \mathpzc{q}_l^i)$ where $j_i$ is the vertex that is missing in the corresponding quiver. Contrary to (ii) it is not guaranteed that the bad paths in the corresponding algebras $eAe$ for the same vertex $l$ but different $J(l)$ coincide, so we have to do the following case-by-case-analysis. Assume without loss of generality that $\sigma_l^4(a)\neq \tau_l^4(x)$, so that $\mathpzc{q}_l^4(a)\mathpzc{p}_l^4(x)=0$, otherwise interchange the rôles of $x$ and $y$. If $\sigma_l^3(b)\neq \tau_l^3(y)$, then define $\tau(x):=\tau_l^4(x)$, $\tau(y):=\tau_l^4(y)$, $\sigma(a):=\sigma_l^4(a)$, $\sigma(b):=-\sigma_l^4(a)$, $\mathpzc{p}(x):=\mathpzc{p}_l^4(x)$, $\mathpzc{q}(a):=\mathpzc{q}_l^4(a)$, $\mathpzc{p}(y):=\mathpzc{p}_l^3(y)$ and $\mathpzc{q}(b):=\mathpzc{q}_l^3(b)$. Otherwise we have $\mathpzc{q}_l^3(b)\mathpzc{p}_l^3(x)=0$. In that case if $\sigma_l^1(b)\neq \tau_l^1(y)$, then take $\tau(x):=\tau_l^4(x)$, $\tau(y):=\tau_l^4(y)$, $\sigma(a):=\sigma_l^4(a)$, $\sigma(b):=-\sigma_l^4(a)$, $\mathpzc{p}(x):=\mathpzc{p}_l^4(x)$, $\mathpzc{q}(a):=\mathpzc{q}_l^4(a)$, $\mathpzc{p}(y):=\mathpzc{p}_l^1(y)$ and $\mathpzc{q}(b):=\mathpzc{q}_l^1(b)$. Otherwise in that case we also have $\sigma_l^1(a)\neq \tau_l^1(y)$ and we can then define $\tau(x):=\tau_l^3(x)=:\sigma(a)$, $\tau(y):=\tau_l^3(y)=:\sigma(b)$, $\mathpzc{p}(x):=\mathpzc{p}_l^3(x)$, $\mathpzc{q}(b):=\mathpzc{q}_l^3(x)$, $\mathpzc{p}(y):=\mathpzc{p}_l^1(y)$ and $\mathpzc{q}(a):=\mathpzc{q}_l^1(a)$. In each case we get surjective maps $\mathpzc{p}, \mathpzc{q}$, and hence a bisected presentation, with the same argument as for (ii).
\end{proof}

\begin{rmk}
\begin{enumerate}[(a)]
\item One can get rid of the assumption that the algebra has to be basic by adjusting the definition of $J(l)$ by taking at least one representative of any isomorphism class $[Ae_i]$.
\item Note that for non-biserial algebras with biserial quiver the algebras $eAe$ do in general not have biserial quiver.
\item For special biserial algebras, it is possible to go from $A$ to $eAe$ and staying special biserial. However one cannot go back, as one can see from the example in \cite{SW83} of a biserial algebra which is not a special biserial algebra. For idempotents as described in the theorem $eAe$ is always special biserial.
\item That fewer points than in (iii) are not sufficient for testing biseriality is already appearent for the path algebra of $\mathbb{D}_4$: If we take only two neighbours we get the path algebra of $\mathbb{A}_3$, which is obviously Nakayama, and therefore biserial.
\item One reason why one can also not get rid of multiple arrows in general is the same as for assumption (C2) in \ref{corVF}, for example take the quiver $\begin{xy}\xymatrix{1\ar@<2pt>[r]^x\ar@<-2pt>[r]_y &2\ar@<2pt>[r]^a\ar@<-2pt>[r]_b&3}\end{xy}$ with relations $(a-b)x$ and $(b-a)y$, which is not biserial, but subalgebras with fewer arrows are biserial.
\end{enumerate}
\end{rmk}

\section{$\mathbb{D}_4$-free algebras}

In this section we present our new description of basic biserial algebras, namely $\mathbb{D}_4$-free algebras, and prove that the two defintions coincide. As a corollary we get a description of biseriality in terms of the subalgebras mentioned in Theorem \ref{5pt}.

\begin{defn}\label{D4-free}
Let $A$ be a basic algebra with a complete set of primitive othogonal idempotents $\{e_1,\dots,e_n\}$. Then $A$ is called \emphbf{$\mathbb{D}_4$-free}, if there does not exist one of the following modules:
\begin{enumerate}[(1)]
\item a local module $M$ of Loewy length two with $l(\rad M)=3$,
\item a colocal module $M$ of Loewy length two with $l(M/\soc M)=3$,
\item a local module $M$, indices $i,j\in \{1,\dots,n\}$, $\tilde{a}_1\in e_i\rad{A}e_j$, $\tilde{a}_2, \tilde{a}_3\in \rad{A}$, $b_0\in M$, such that
\begin{enumerate}[(a)]
\item $\tilde{a}_2\tilde{a}_1b_0, \tilde{a}_3\tilde{a}_1b_0$ are linearly independent
\item $\radsquare{A}\tilde{a}_1b_0=0$
\item there do not exist $\hat{a}_1,\hat{a}_1'\in e_i\rad{A}e_j$, $\hat{a}_2,\hat{a}_3\in \rad{A}$ such that
\begin{enumerate}
\item[($\alpha$)] $\overline{\hat{a}_2}, \overline{\hat{a}_3}\in \langle \tilde{a}_2,\tilde{a}_3\rangle_k/(\radsquare{A}\cap \langle \tilde{a}_2,\tilde{a}_3\rangle_k)$ linearly independent
\item[($\beta$)] $\hat{a}_1b_0+\hat{a}_1'b_0=\tilde{a}_1b_0$
\item[($\gamma$)] $\hat{a}_2\hat{a}_1'b_0=0$ and $\hat{a}_3\hat{a}_1b_0=0$.
\end{enumerate}
\end{enumerate}
\item a local right module $M$, indices $i,j\in \{1,\dots,n\}$, $\tilde{a}_1\in e_j\rad{A}e_i$, $\tilde{a}_2, \tilde{a}_3\in \rad{A}$, $b_0\in M$, such that
\begin{enumerate}[(a)]
\item $b_0\tilde{a}_1\tilde{a}_2, b_0\tilde{a}_1\tilde{a}_3$ are linearly independent
\item $b_0\tilde{a}_1\radsquare{A}=0$
\item there do not exist $\hat{a}_1, \hat{a}_1'\in e_j\rad{A}e_i$, $\hat{a}_2, \hat{a}_3\in \rad{A}$ such that
\begin{enumerate}
\item[($\alpha$)] $\overline{\hat{a}_2}, \overline{\hat{a}_3}\in \langle \tilde{a}_2, \tilde{a}_3\rangle_A/(\radsquare{A}\cap \langle \tilde{a}_2,\tilde{a}_3\rangle_A)$ linearly independent
\item[($\beta$)] $b_0\hat{a}_1+b_0\hat{a}_1'=b_0\tilde{a}_1$
\item[($\gamma$)] $b_0\hat{a}_1'\hat{a}_2=0$ and $b_0\hat{a}_1\hat{a}_3=0$.
\end{enumerate}
\end{enumerate}
\end{enumerate}
\end{defn}

\begin{rmk}
An algebra $A$ is biserial iff its opposite algebra $A^{op}$ is biserial. $A$ is also $\mathbb{D}_4$-free iff $A^{op}$ is.\\
The reader may have noticed, that (3) and (4) do not necessarily describe ``small'' modules in the sense that their length or Loewy length is bounded but instead give some condition on a ``small'' part of a possibly ``large'' module (cf. (b)). This is because of the path algebra of the following quiver (and similar ones):
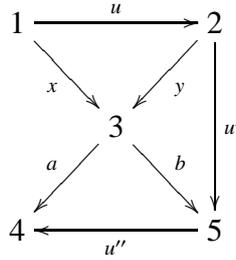
\[\begin{xy}\xymatrix{1\ar[rr]^u\ar[rd]_x&&2\ar[ld]^y\ar[dd]^{u'}\\&3\ar[rd]^b\ar[ld]_a\\4&&5\ar[ll]^{u''}}\end{xy}\]
with relations $ax=u''u'u$ and $by$, $yu$, $u''b$. If we want to have a module with similar properties as in (3) but replacing (b) with $\radoperator^3(A)M=0$, for example $P_1/\radoperator^3(A)P_1$, then this would be a module over the string algebra with the same quiver and relations $ax$ and $by$, $yu$, $u''b$.
\end{rmk}

\begin{lem}\label{lem2}
Let $A$ be an algebra.
\begin{enumerate}[(i)]
\item There is a local module $M$ of Loewy length two with $l(\rad M)=m$ iff there is a point in the quiver of $A$ where $m$ arrows start.
\item There is a colocal module $M$ of Loewy length two with $l(M/\soc M)=m$ iff there is a point in the quiver of $A$ where $m$ arrows end.
\end{enumerate}
\end{lem}

\begin{proof}
\begin{enumerate}[(i)]
\item Without loss of generality let $A=kQ/I$ for some quiver $Q$ and an admissible ideal $I$, since both conditions hold true iff they hold true for the corresponding basic algebra and any basic algebra is of that form.
\begin{enumerate}
\item[``$\Leftarrow$'':] Let $i$ be the point where $m$ arrows start, then $M:=Ae_i/\radsquare{A}e_i$ is a local module with $l(\rad M)\geq m$ and a factor module of it has the required properties.
\item[``$\Rightarrow$'':] Let $M$ be such a module. Let $b_0\in M$, s.t. $\overline{b_0}$ spans $\Kopf M$. Since $M$ is a local module, there exists $e_j$, s.t. $\overline{e_jb_0}$ also spans $\Kopf M$. $\rad M=\rad{A}\cdot M$ and since $M$ is of Loewy length two and $l(\rad M)=m$ there exist $a_1,\dots,a_m\in Q_1$ with $a_1e_jb_0, \dots,a_me_jb_0$ linearly independent, as a consequence they all start in the vertex $j$.
\end{enumerate}
\item This is the dual statement to (i).
\end{enumerate}
\end{proof}

\begin{thm}
Let $A$ be a basic algebra with complete set of primitive orthogonal idempotents $\{e_1,\dots,e_n\}$. Then $A$ is $\mathbb{D}_4$-free iff it is biserial.
\end{thm}

\begin{proof}
Assume $A$ is biserial, then Lemma \ref{lem2} for $m=3$ shows that modules of the form (1) and (2) do not exist. As (4) is dual to (3) it remains to prove (3).\\
Suppose to the contrary that a module of the form (3) with vertices $i$ and $j$ and the required elements exists. According to Corollary \ref{corVF} we may assume that $A=kQ/I$ satisfies the conditions stated there. Since $Q$ is a biserial quiver there end at most two arrows $a_1, a_1'$ in the vertex $i$ (define $a_1':=0$ if there does not exist a second arrow ending in $i$) and we can decompose $\tilde{a}_1=a_1p+a_1'p'$ with $p,p'\in kQ$. We may assume without loss of generality that $\tilde{a}_2=\mu_2a_2+\mu_3a_3+r$ and $\tilde{a}_3=\mu_2'a_2+\mu_3'a_3+r'$, where $a_2,a_3\in Q_1$ with $s(a_2)=s(a_3)=i$ and $r,r'\in \radsquare{A}e_i$. Otherwise we can replace $\tilde{a}_2$ and $\tilde{a}_3$ by $\tilde{a}_2e_i$ and $\tilde{a}_3e_i$ and $r,r'$ by $re_i$ and $r'e_i$ and get elements with the same properties. Define $\hat{a}_1:=a_1p$, $\hat{a}_1':=a_1'p'$. One of the paths $a_2a_1$ and $a_3a_1$ is bad, assume without loss of generality, that it is $a_2a_1$. If $\hat{a}_1'b_0=0$, then $\tilde{a}_2\tilde{a}_1b_0$ and $\tilde{a}_3\tilde{a}_1b_0$ are not linearly independent because the necessary relation $(a_2-\omega qa_3)a_1$, $\omega\in k$, $q$ a path in $Q$, possibly a zero path, yields $\tilde{a}_2\tilde{a}_1b_0, \tilde{a}_3\tilde{a}_1b_0\in \langle a_3a_1pb_0\rangle_k$. If both $\hat{a}_1b_0$ and $\hat{a}_1'b_0$ are non-zero, then there are two necessary relations $(a_2-\omega qa_3)a_1$ and $(a_3-\kappa q'a_2)a_1'$. The elements $a_2-\omega qa_3$ and $a_3-\kappa q'a_2$ are linearly independent modulo $\radsquare{A}$, either because the ideal is admissible or because of (C2) in Corollary \ref{corVF}, so the elements $\hat{a}_1:=a_1p$, $\hat{a}_1':=a_1'p'$, $\hat{a}_2:=a_2-\omega qa_3$ and $\hat{a}_3:=a_3-\kappa q'a_2$ define elements contradicting condition (c) on the module (3).\\
For the converse suppose that $A$ is a non-biserial algebra. If the quiver of $A$ is non-biserial, then according to Lemma \ref{lem2} there does exist a module of the form (1) or (2). So suppose that the quiver of $A$ is biserial. Then for every quadruple $(\sigma,\tau,\mathpzc{p},\mathpzc{q})$, where $(\sigma, \tau)$ is a bisection and $\mathpzc{p}, \mathpzc{q}$ are surjective algebra homomorphisms $kQ\to A$ with $\mathpzc{p}(e_i)=\mathpzc{q}(e_i)$ and $\mathpzc{p}(a),\mathpzc{q}(a)\in \rad{A}$ for every arrow $a\in Q_1$, there exist arrows $a,x\in Q_1$ such that $\mathpzc{q}(a)\mathpzc{p}(x)\neq 0$. We prove that in this case there is a module $M$ with properties (a)-(c) by analyzing the local situation at the vertex $s(a)=t(x)$ and redefining the values of $\sigma$ and $\mathpzc{q}$ (resp. $\tau$ and $\mathpzc{p}$) for the arrows starting (resp. ending) at this vertex and getting a bisected presentation if there is no such module $M$. We say that $(Q,\sigma,\tau,\mathpzc{p},\mathpzc{q})$ is a bisected presentation at a vertex $l$ if for all bad paths $ax$ of length two with $s(a)=t(x)=l$, $\mathpzc{q}(a)\mathpzc{p}(x)=0$.\\
There are six possible local situations: One arrow starts at this vertex but none ends, none ends but one arrow starts, one arrow starts and one arrow ends, two arrows start at this vertex but only one ends, only one starts but two end, or two arrows start and two end. In the first three instances we define all paths to be good. Then any surjective algebra homomorphism will give rise to a bisected presentation.\\
For the case that two arrows $a,b$ are starting but only the arrow $x$ is ending we can assume that also $\mathpzc{q}(b)\mathpzc{p}(x)\neq 0$, otherwise we could interchange $\sigma(a)$ and $\sigma(b)$ to get a bisected presentation at this point. Now look at the module $M:=Ae_{s(x)}/\radsquare{A}\mathpzc{p}(x)$ and at the elements $b_0:=\overline{e_{s(x)}}, \tilde{a}_1:=\mathpzc{p}(x), \tilde{a}_2:=\mathpzc{q}(a), \tilde{a}_3:=\mathpzc{p}(a)$. If $\overline{\mathpzc{q}(a)\mathpzc{p}(x)}$ and $\overline{\mathpzc{q}(b)\mathpzc{p}(x)}$ were linearly dependent, then without loss of generality $\mathpzc{q}(a)\mathpzc{p}(x)+\lambda \mathpzc{q}(b)\mathpzc{p}(x)=r\mathpzc{p}(x)$ with $r\in \radsquare{A}\mathpzc{p}(x)$ and $\lambda\in k$. We can assume that $r\in e_{t(a)}Ae_{s(a)}$. We then redefine $\mathpzc{q}'(a):=\mathpzc{q}(a)+\lambda\mathpzc{q}(b)-r$. Leaving everything else unchanged we get an algebra homomorphism because all elements lie in $e_{t(a)}Ae_{s(a)}$. Its surjectivity follows from \cite{B95} Proposition 1.2.8 as we have modified by an element in $\radsquare{A}$. So we get a bisected presentation at this point. We now have found a module with (a) and (b) satisfied but we also have to prove that (c) holds. Therefore suppose that there are elements $\hat{a}_1,\hat{a}_1', \hat{a}_2, \hat{a}_3$ as in (c). Then one of the elements $\overline{\hat{a}_1}, \overline{\hat{a}_1'}$ has to span $e_{t(x)}\rad{A}/\radsquare{A}e_{s(x)}$, without loss of generality it is $\hat{a}_1$, then redefine $\mathpzc{p}'(x):=\hat{a}_1, \mathpzc{q}'(a)=\hat{a}_3, \mathpzc{q}'(b)=\hat{a}_2$ and get a bisected presentation at this point.\\
For the case that only one arrow is starting at this point but two are ending proceed dually. Note that if $(Q,\sigma,\tau,\mathpzc{p},\mathpzc{q})$ is a bisected presentation for $A$, then $(Q^{op},\tau,\sigma,\mathpzc{q},\mathpzc{p})$ is a bisected presentation for $A^{op}$.\\
So suppose that there are two arrows $a,b$ starting and two, $x,y$, ending at this point. First we want to achieve that for some combination of two arrows, $\mathpzc{q}(a)\mathpzc{p}(x)=0$. Look at the module $M:=Ae_{s(x)}/\radsquare{A}\mathpzc{p}(x)$ and the elements $b_0:=\overline{e_{s(x)}}$, $\tilde{a}_1:=\mathpzc{p}(x)$, $\tilde{a}_2:=\mathpzc{q}(a)$, $\tilde{a}_3:=\mathpzc{q}(b)$. This module and the elements satisfy (b). Assume it does not satisfy (a), i.e. $\overline{\mathpzc{q}(a)\mathpzc{p}(x)}$ and $\overline{\mathpzc{q}(b)\mathpzc{p}(x)}$ are linearly dependent. Then without loss of generality $\mathpzc{q}(a)\mathpzc{p}(x)+\lambda\mathpzc{q}(b)\mathpzc{p}(x)=r\mathpzc{p}(x)$ otherwise interchange the r\^oles of $a$ and $b$. Then define $\mathpzc{q}'(a):=\mathpzc{q}(a)+\lambda\mathpzc{q}(b)-r$ and achieve $\mathpzc{q}'(a)\mathpzc{p}(x)=0$. So assume this module does not satisfy (c), then there exist $\hat{a}_1,\hat{a}_1', \hat{a}_2,\hat{a}_3$ with the required properties. Because of ($\beta$) $\overline{\hat{a}_1}$ or $\overline{\hat{a}_1'}$ has to span (sometimes together with $\overline{\mathpzc{p}(y)}$) $e_{t(x)}\rad{A}/\radsquare{A}e_{s(x)}$, assume without loss of generality it is $\hat{a}_1$. Furthermore we have $e_{t(a)}\hat{a}_2e_{s(a)}\neq 0$ or $e_{t(b)}\hat{a}_2e_{s(b)}\neq 0$ and the other way round for $\hat{a}_3$, without loss of generality it is the former. Thus we can define $\mathpzc{q}'(a):=\hat{a}_3$, $\mathpzc{q}'(b):=\hat{a}_2$ and $\mathpzc{p}'(x):=\hat{a}_1$ to achieve $\mathpzc{q}'(a)\mathpzc{p}'(x)=0$.\\
So from now on we can assume that $\mathpzc{q}(a)\mathpzc{p}(x)=0$, otherwise we would have a module of the form (3). Now look at the right module $M:=e_{t(b)}A/\mathpzc{q}(b)\radsquare{A}$, and the elements analogous to the above arguments. Assume $\overline{\mathpzc{q}(b)\mathpzc{p}(x)}$ and $\overline{\mathpzc{q}(b)\mathpzc{p}(y)}$ are linearly dependent. Then we have $\lambda_1\mathpzc{q}(b)\mathpzc{p}(x)+\lambda_2\mathpzc{q}(b)\mathpzc{p}(x)=\mathpzc{q}(b)r'$ with $r'\in \radsquare{A}$. If $\lambda_2\neq 0$, we can define $\mathpzc{p}'(y):=\lambda_2\mathpzc{p}(y)+\lambda_1\mathpzc{p}(x)-r'$ to get a bisected presentation at this point with bad paths $ax$ and $by$. If on the other hand $\lambda_2=0$, then we also have to look at the module $M':=Ae_{s(y)}/\radsquare{A}\mathpzc{p}(y)$ with analogous elements. If $\overline{\mathpzc{q}(a)\mathpzc{p}(y)}$ and $\overline{\mathpzc{q}(b)\mathpzc{p}(y)}$ are linearly dependent in this module, then $\mu_1\mathpzc{q}(a)\mathpzc{p}(y)+\mu_2\mathpzc{q}(b)\mathpzc{p}(y)=r''\mathpzc{p}(y)$ for some $r''\in \radsquare{A}$. If $\mu_2\neq 0$, then we can define $\mathpzc{q}'(b):=\mu_2\mathpzc{q}(b)+\mu_1\mathpzc{q}(a)-r''$ and we have a bisected presentation at this point with bad paths $ax$ and $by$. If otherwise $\mu_2=0$, then we can redefine $\mathpzc{q}'(a):=\mu_1\mathpzc{q}(a)-r''$ and $\mathpzc{p}'(x):=\lambda_1\mathpzc{p}(x)-r'$ to get a bisected presentation at this point with bad paths $ay$ and $bx$. So $M'$ satisfies (a) and (b). Assume it does not satisfy (c), so there exist elements $\hat{a}_1,\hat{a}_1',\hat{a}_2,\hat{a}_3$ with the required properties. As above one of $\overline{\hat{a}_1}, \overline{\hat{a}_1'}$ (sometimes together with $\overline{\mathpzc{p}(x)}$) does span $e_{t(y)}\rad{A}/\radsquare{A}e_{s(y)}$, without loss of generality assume it is $\hat{a}_1$. Now there are two cases: If $e_{t(b)}\hat{a}_2e_{s(b)}$ is linearly independent of $\mathpzc{q}(a)$ modulo $\radsquare{A}$, then we can define $\mathpzc{q}'(b):=\hat{a}_2$ and $\mathpzc{p}'(y):=\hat{a}_1$ to get a bisected presentation with bad paths $ax$ and $by$. If  this is not the case, then $e_{t(a)}\hat{a}_2e_{s(a)}$ is linearly independent of $\mathpzc{q}(b)$ modulo $\radsquare{A}$ and we can define $\mathpzc{q}'(a):=\hat{a}_2, \mathpzc{p}'(x):=\lambda_1\mathpzc{p}(x)-r', \mathpzc{p}'(y):=\hat{a}_1$ to get a bisected presentation with bad paths $ay$ and $bx$.\\
Now we have shown, that for $M$ the conditions (a) and (b) hold or there exists a module of the form (3) or (4). So assume $M$ does not satisfy (c). Again we have that one of the elements $\hat{a}_1, \hat{a}_1'$ (sometimes together with $\mathpzc{q}(a)$) spans $e_{t(b)}\rad{A}e_{s(b)}$ modulo $\radsquare{A}$, without loss of generality assume again it is $\hat{a}_1$. Dual to what we have done there are two cases: If $e_{t(x)}\hat{a}_2e_{s(x)}$ is linearly independent of $\mathpzc{p}(x)$ modulo $\radsquare{A}$, then we can redefine $\mathpzc{p}'(y):=\hat{a}_2$ and $\mathpzc{q}'(b):=\hat{a}_1$ to get a bisected presentation at this point with bad paths $ax$ and $by$. If this is not the case, then $e_{t(y)}\hat{a}_2e_{s(y)}$ is linearly independent of $\mathpzc{p}(y)$ modulo $\radsquare{A}$. We can now redefine $\mathpzc{q}'(b):=\hat{a}_1$ and in the following we can either assume that $\mathpzc{q}(a)\mathpzc{p}(x)=0$ or by redefining $\mathpzc{q}'(b):=\hat{a}_1$ that $\mathpzc{q}(b)\mathpzc{p}(x)=0$.\\
We have to look at one last module, namely $M':=Ae_{s(y)}/\radsquare{A}\mathpzc{p}(y)$. If this module does not satisfy (a), i.e. $\kappa_1\mathpzc{q}(a)\mathpzc{p}(y)+\kappa_2\mathpzc{q}(b)\mathpzc{p}(y)=r'''\mathpzc{p}(y)$, then we can without loss of generality assume that $\kappa_2\neq 0$, so that we can redefine $\mathpzc{q}'(b):=\kappa_2\mathpzc{q}(b)+\kappa_1\mathpzc{q}(a)-r'''$ to get a bisected presentation at this point with bad paths $ax$ and $by$, otherwise we would use the redefinition as above that $\mathpzc{q}(b)\mathpzc{p}(x)=0$ and redefine $\mathpzc{q}(a)$ to get a bisected presentation at this point with bad paths $ay$ and $bx$. So we can assume that $M'$ satisfies (a) and (b). Assume it does not satisfy (c). Then again we can assume that we can redefine $\mathpzc{p}'(y):=\hat{a}_1$ and either $\mathpzc{q}'(b):=\hat{a}_2$ or $\mathpzc{q}'(a):=\hat{a}_3$ to get a bisected presentation at this point (bad paths are either $ax$ and $by$ or $ay$ and $bx$).
\end{proof}

Out of the proof we get the following corollary:

\begin{cor}
If $A=kQ/I$ is an algebra, where $Q$ is biserial, such that $eAe$ has no oriented cycles for any idempotent as in theorem \ref{5pt} (iii), then $A$ is biserial iff for all idempotents $e$ as in theorem \ref{5pt} (iii) there does not exist a local $eAe$-module $M$, such that there exists $\tilde{b}_1\in e_l\rad{M}$ with $l(\rad{A}\tilde{b}_1)\geq 2$ and $\radsquare{A}\tilde{b}_1=0$ and there does not exist a colocal $eAe$-module $M$, such that there exists $\tilde{b}_1\in e_lM\setminus \soc(M)$ with $l(A\tilde{b}_1/\soc(A\tilde{b}_1))\geq 2$ and $\soc^2(A\tilde{b}_1)=A\tilde{b}_1$ or $eAe$ is isomorphic to one of the following string algebras with quiver
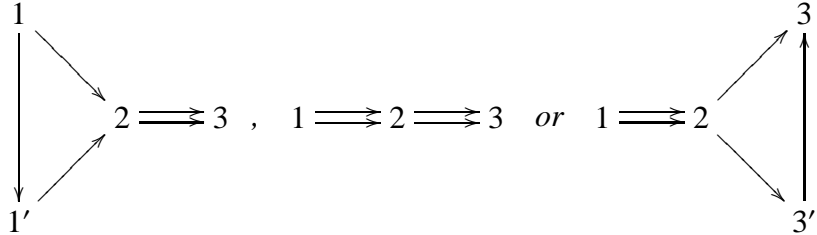
\begin{center}
\begin{tabular}[c]{c}$\begin{xy}\xymatrix{1\ar[rd]\ar[dd]\\&2\ar@<2pt>[r]\ar@<-2pt>[r]&3\\1'\ar[ru]}\end{xy}$\end{tabular}, \begin{tabular}[c]{c}$\begin{xy}\xymatrix{1\ar@<2pt>[r]\ar@<-2pt>[r]&2\ar@<2pt>[r]\ar@<-2pt>[r]&3}\end{xy}$\end{tabular} or \begin{tabular}[c]{c}$\begin{xy}\xymatrix{&&3\\1\ar@<2pt>[r]\ar@<-2pt>[r]&2\ar[ru]\ar[rd]\\&&3'\ar[uu]}\end{xy}$\end{tabular}
\end{center}
\end{cor}

\begin{proof}
If $A$ is a biserial algebra such that $eAe$ has no oriented cycles for any $e$, then the module $M$ defined in the proof satisfies conditions (a) and (b) and therefore has the properties mentioned in the corollary with $\tilde{b}_1:=\tilde{a}_1b_0$. If it does not satisfy (c), then we have defined in the proof above elements $\hat{a}_2$ and $\hat{a}_3$ which span $e_{t(a_3)}Ae_{t(a_2)}$, so if we take the isomorphism mapping $ \hat{a}_2\mapsto a_2$ and $\hat{a}_3\mapsto a_3$, then we obtain one of the exceptional string algebras.\\
In the reverse direction of the proof the converse is also proven because a module $M$ with properties (a) and (b) is constructed there and therefore also satisfies the conditions of the corollary.
\end{proof}

\section*{Acknowledgement}
The results of this article are part of my diploma thesis written in 2009 at the University of Bonn. I would like to thank my advisor Jan Schr\"oer for helpful discussions and continous encouragement. I would also like to thank Rolf Farnsteiner for his comments on a previous version of this paper.

\bibliographystyle{alpha}
\bibliography{publication}

\end{document}